\documentclass[12pt,a4paper]{amsart}
\usepackage{amsmath}
\usepackage{amsthm}
\usepackage{amssymb}
\usepackage{mathrsfs}
\usepackage[english]{babel}
\usepackage{ot1patch}

\newtheorem{thm}{Theorem}[section]
\newtheorem{lem}[thm]{Lemma}
\newtheorem{cor}[thm]{Corollary}
\newtheorem{prop}[thm]{Proposition}

\theoremstyle{remark}
\newtheorem{rem}[thm]{Remark}
\newtheorem*{rem*}{Remark}

\theoremstyle{definition}

\newtheorem{ex}[thm]{Example}

\numberwithin{equation}{section}

\newcommand{\C}{\mathbb{C}}

\begin{document}

\title{Multiplicity and the pull-back problem}

\author{Maciej P. Denkowski}\address{Jagiellonian University, Faculty of Mathematics and Computer Science, Institute of Mathematics, \L ojasiewicza 6, 30-348 Krak\'ow, Poland}\email{maciej.denkowski@uj.edu.pl}
\date{March 24th 2014}

\begin{abstract}
We discuss a formula of S. Spodzieja and generalize it for the isolated improper Achilles-Tworzewski-Winiarski intersection index. As an application we give a simple proof of a result of P. Ebenfelt and L. Rothschild: if $F\colon ({\C}^m,0)\to ({\C}^m,0)$ is a finite holomorphic map, $W$ a germ of a complex variety at zero such that $F^{-1}(W)$ is a smooth germ and the Jacobian of $F$ does not vanish identically on it, then $W$ is smooth too. 
\end{abstract}

\maketitle

\section{Introduction}

One of the main results of the article \cite{S} can be stated as follows:
\begin{thm}[Spodzieja]\label{Spodzieja}
If $f\colon D\to f(D)\subset {\C}^n$ is a holomorphic branched covering on a domain $D\subset{\C}^m$ with $f^{-1}(0)=\{0\}$, then 
$$
i(\Gamma_f\cdot(D\times\{0\}^n);0)=\mathrm{deg}_0 f(D)\cdot \tilde{m}_0(f)\leqno{(S)}
$$
where $\Gamma_f$ denotes the graph of $f$, $i(\Gamma_f\cdot(D\times\{0\}^n);0)$ is the Achilles-Tworzewski-Winiarski \textit{intersection index} \cite{ATW} at the origin, $\mathrm{deg}_0 f(D)$ stands for the \textit{local degree}, i.e. the usual Lelong number of $f(D)$ at zero, and $\tilde{m}_0(f)$ is a type of \textit{geometric multiplicity} of $f$.
\end{thm}

Note that $f$ being proper we necessarily have $n\geq m$. If $n>m$,then the isolated intersection $\Gamma_f\cap(D\times\{0\}^n)$ is not proper (i.e. the codimensions in $D\times{\C}^n$ does not add to $m+n$) and the intersection index cannot be computed along Draper \cite{Dr}; instead we use \cite{ATW}. Of course, if $n=m$, then $\mathrm{deg}_0 f(D)=1$ and the (usual) geometric multiplicity $m_0(f)$, which in this case is the generic number of points in the fibre, coincides with the intersection index. For branched coverings we refer the reader to \cite{Ch}. Proper intersection theory can be found in \cite{Dr} and \cite{Ch}.

The \textit{geometric multiplicity} is usually understood (e.g. see \cite{Ch}, \cite{L}) as
$$
m_0(f)=\limsup_{y\to 0}\#(f^{-1}(y)\cap U)
$$
where $U$ is a sufficiently small neighbourhood of zero. 

However, it should be pointed out that it is not this multiplicity that is used in the aforementioned 
Theorem \ref{Spodzieja} (that explains the tilda). Unfortunately, this may not be really apparent from the original article. Example \ref{ex} shows that the formula $(S)$ with $m_0(f)$ is indeed  erroneous. 
  Namely, the geometric multiplicity $m_0(f)$ is too big a number and should be replaced by 
$$
\tilde{m}_0(f)=\limsup_{\mathrm{Reg} f(D)\ni y\to 0}\#(f^{-1}(y)\cap U).
$$
Note that by the Remmert Proper Mapping Theorem, $f(D)$ is an analytic set. Since it is irreducible, $\mathrm{Reg}f(D)$ is connected and $f|_{f^{-1}(\mathrm{Reg}f(D))}$ is a branched covering over it, with a well-defined covering number. Since Theorem \ref{Spodzieja} is a very nice result on which many other papers are based (among them one of ours), it seems natural to stress the importance of the proper definition of the multiplicity used in it. 

In the next section we will prove a more general result, Theorem \ref{wzor}, that implies the theorem above, and we will apply it to the so-called `pull-back problem' (actually, this was precisely the question that triggered off our research and led to finding a counter-example to $(S)$ with $m_0(f)$ instead of $\tilde{m}_0(f)$). Namely, we will give a simple proof of the following theorem: 
\begin{thm}[Ebenfelt-Rothschild]\label{pullback}
 Let $F\colon({\C}^m, 0)\to ({\C}^m, 0)$ be a finite holomorphic mapping and $V\subset{\C}^m$ an analytic set germ at zero. Assume that   $V=F^{-1}(F(V))$ and $V$ is a smooth germ. If $\mathrm{Jac} F|_V\not\equiv 0$, then $F(V)$ is smooth too.
\end{thm}
Actually, this result is a by-product of the main theorem of \cite{ER} concerning images of real-analytic submanifolds by holomorphic finite mappings. However, the question whether one can omit the assumption on the Jacobian in the theorem above is stated there (Remark 2.2 and Question starting section 4 in \cite{ER}) and is obviously of interest and remains open. In \cite{ER} it is proved that the assumption on the Jacobian can be dropped in the case of a curve. 

This question was the starting point of another article, very nice indeed, \cite{L} by J. Lebl where among others the theorem of Ebenfelt and Rothschild for $\dim V=1$ (without the assumption on the Jacobian) is obtained in a simpler way. Actually, the author tackles a more general problem, i.e. he asks which properties of a germ $W\subset{\C}^m$ at zero are inherited from the pull-back $V:=F^{-1}(W)$ (note that $W=F(V)$, $F$ being the germ of a branched covering). The simplest example is, of course, that of irreducibility. A somewhat less obvious one is of normality, and Lebl gives a short and elegant proof. But it is smoothness which is the central subject. For curves normality is equivalent to smoothness. Many different instances of the theorem are discussed in a most accessible way, and the result is proved without the assumption on the Jacobian in some special cases (such as when the multiplicity of $F$ is a prime number). 

It is quite natural to expect that the theorem of Ebenfelt and Rothschild should hold true without any assumption on the Jacobian. 

The important thing here that makes things work in the Ebenfelt-Rothschild result is that $V$ is a `plain' pull-back of $W$. It is easy to find a finite polynomial mapping $F\colon{\C}^m\to{\C}^m$ and an algebraic smooth set $V$ such that $F(V)$ is algebraic and singular (but in this case $V\subsetneq F^{-1}(F(V))$). To that purpose the following example was devised with Carlo Perrone:
\begin{ex}
Let $F(x,y)=(x^2,y)$ and $W=\{u^3=v^2\}$ with a singularity at zero. Then $F^{-1}(W)=\{x^6=y^2\}$ is a reducible curve ($\Gamma_\pm:=\{(\pm t, t^3), t\in{\C}\}$ are two different components of $F^{-1}(W)$ depending on the choice of the sign). Let $V:=\Gamma_+$. It is a non-singular curve and clearly $F(V)=W$. 
\end{ex}

We have the feeling that our approach, using intersection theory, should shed a new light on the pull-back problem.

\noindent\textbf{Notation.}
We denote by $\mathrm{Reg} A$ the regular part of an analytic set $A$ and we put $\mathrm{Sng} A:=A\setminus\mathrm{Reg} A$. 

In order to shorten notation, we will write $$i_0(f):=i(\Gamma_f\cdot (D\times\{0\}^n);0).$$

\section{A general intersection formula for proper projections}

We start this section by giving a counter-example to formula $(S)$ with $m_0(f)$ instead of $\tilde{m}_0(f)$.
\begin{ex}\label{ex}
Consider the set $X=\{(x,y,t)\in{\C}^3\mid y^2=x^2(x+t^2)\}$. It is a family in $t$ of globally irreducible algebraic curves, each of which can be normalized by the parametrization $s\mapsto (s^2-t^2,s(s^2-t^2))$. Consider the graph of the resulting holomorphic function $g(s,t)$ written as
$$
A:=\{(s, s^2-t^2,s(s^2-t^2),t)\mid (s,t)\in{\C}^2\}\subset{\C}^4
$$
together with the projection $\pi(z,x,y,t)=(x,y,t)$. Then $\pi(A)=X$ and it is easy to check that $\pi|_A$ is proper. 

Clearly, $\pi^{-1}(0,0,0)\cap A=\{(0,0,0,0)\}$ and by composing $\pi$ with the natural parametrization $\gamma(s,t)=(s,g(s,t),t)$ of $A$ we obtain a holomorphic proper function $f=\pi\circ\gamma\colon {\C}^2\to{\C}^3$ with $f^{-1}(0)=\{0\}$, and $f({\C}^2)=X$. We have
$$
f(s,t)=(s^2-t^2,s(s^2-t^2),t).
$$

Now, $X$ being a hypersurface, we easily check that the singular part $\mathrm{Sng} X$ is the $t$-axis, as it is given by the equations $F=0,\nabla F=0$ for $F(x,y,t)=y^2-x^2(x+t^2)$ forming a nowherdense subset of $X$ (\footnote{cf. Tsikh's result in \cite{Ch}.}). Moreover, $F$ is a minimal defining function for $X$, and so $\mathrm{deg}_0 X=\mathrm{ord}_0 F=2$ (\footnote{Otherwise, we may simply say that the tangent cone $C_0(X)$ is described by the initial form of the expansion of $F$, so it is the $(x,t)$-plane. The projection on this plane realizes the degree at zero and it is clearly two-sheeted.}). 

Over the regular points $\mathrm{Reg} X$, the mapping $f$ is one-to-one, but $f$ as a covering $f^{-1}(\mathrm{Sng} X)\to \mathrm{Sng} X$ is two-sheeted. Indeed, it is easy to see that $f^{-1}(\mathrm{Sng} X)=\{(s,\pm s)\mid s\in{\C}^2\}$ and $f(s,\pm s)=(0,0,\pm s)$ is a double cover of the $t$-axis.

Therefore, $m_0(f)=2$, whereas $\tilde{m}_0(f)=1$. In order to compute $i_0(f)$ we need only (thanks to the beautiful result of \cite{S}, independent of the formula $(S)$) to compose $f$ with a general linear projection (realizing $\mathrm{deg}_0X$), say $p(x,y,t)=(x,t)$, and compute $m_0(p\circ f)$. Since 
$$
(p\circ f)(s,t)=(s^2-t^2,t),
$$
if we take a point $(u,v)$ with $u\neq 0$, from the equation $(p\circ f)(s,t)=(u,v)$ we obtain $t=v$ and $s$ is the solution of $s^2=u+v^2$. Hence $i_0(f)=m_0(p\circ f)=2$ and eventually
$$
i_0(f)=2=\tilde{m}_0(f)\cdot\mathrm{deg}_0 f({\C}^2)\neq {m}_0(f)\cdot\mathrm{deg}_0 f({\C}^2)=4.
$$
\end{ex}

Below, we will give a generalized version of $(S)$.

We recall that the \textit{relative tangent cone} $C_a(A,B)$ of two locally analytic sets $A,B\subset{\C}^M$ at an isolated intersection point $a\in A\cap B$ is defined in \cite{ATW} and coincides with the Peano-Whitney tangent cone $C_0(A-B)$ where $A-B$ is the algebraic difference of $A$ and $B$. It is an algebraic cone of dimension $\dim A+\dim B$ in case $A$ and $B$ are pure dimensional; moreover, if the isolated intersection $A\cap B$ is transverse at $a$, i.e. $C_a(A)\cap C_a(B)=\{0\}$, then $C_a(A,B)=C_a(A)+C_a(B)$ (see \cite{ATW}).

We briefly recall one way of computing the isolated improper intersection index. By the results of \cite{ATW}, if $X,Y$ are locally analytic sets in ${\C}^N$, $X\cap Y=\{0\}$, the intersection is not proper (i.e. $\dim X+\dim Y<N$) and $Y$ is smooth, then the improper intersection index $i(X\cdot Y;0)$ is equal to the proper one $i(X\cdot M;0)$, where $M\supset Y$ is analytic and smooth, $M\cap X=\{0\}$ and $\dim M=N-\dim X$, provided that 
$$
T_0 Y=T_0 M\cap C_0(X,Y).
$$

If $A$ is a pure $k$-dimensional analytic subset of some open set in ${\C}^m\times{\C}^n$ and the projection $p\colon {\C}^m\times{\C}^n\to{\C}^n$ is proper on $A$, then it has a well defined multiplicity as a branched covering over the connected manifold $\mathrm{Reg} X$ for each irreducible component $X\subset p(A)$ ($p(A)$ is analytic by the Remmert Proper Mapping Theorem). We may assume for simplicity that $p(A)$ is irreducible. Then for any $x\in p^{-1}(y)$ where $y\in\mathrm{Sng}p(A)$ we find a neighbourhood $U$ such that $p^{-1}(y)\cap A\cap\overline{U}=\{x\}$ and we define the \textit{regular multiplicity at} $x$ by
$$
\tilde{m}_x(p|_A):=\limsup_{\mathrm{Reg} p(A)\ni z\to y} \#p^{-1}(z)\cap A\cap U.
$$
This is independent of the choice of $U$ and makes sense also at any points in the fibres over $\mathrm{Reg} p(A)$. By the way, it is a classical result of W. Stoll that for any $y\in\mathrm{Reg} p(A)$ there is $$\sum_{x\in p^{-1}(y)\cap A} m_x(p|_A)=m(p|_{A\setminus p^{-1}(\mathrm{Sng} p(A))})\leqno{(St)}$$
where the latter denotes the multiplicity of $p|_{A\setminus p^{-1}(\mathrm{Sng} p(A))}$ as a branched covering over $\mathrm{Reg} p(A)$. 

However, there is no simple relation for the multiplicities above $\mathrm{Sng}p(A)$. In particular, the number of points in the generic fibre of $p|_A$ may exceed the covering mutiplicity over $\mathrm{Reg} p(A)$. Hence, $$\tilde{m}_x(p|_A)\leq m_x(p|_A)$$ 
with equality at least when $k=n$. 
We will need hereafter a the following simple but useful observation.
\begin{lem}\label{wlokno}
Let $D\subset{\C}^n$ be an open set containing zero, $A\subset {\C}^m\times D$ a pure $k$-dimensional analytic set with proper projection $\pi\colon {\C}^m\times D\to D$ and assume that $\pi^{-1}(0)\cap A=\{0\}$. Then for any neighbourhood $U$ of $0\in A$ we are able to find a neighbourhood $V$ of $0\in \pi(A)$ such that for any $y\in U\setminus\Sigma$, $\pi^{-1}(y)\subset U$.
\end{lem}
\begin{proof}
Indeed, if it were not the case, then for some $U$ we would obtain a sequence of points $\pi(A)\ni y_\nu\to 0$ and a sequence of points $x_\nu\in A\setminus U$ with $\pi(x_\nu)=y_\nu$. But then, as $(x_\nu)$ is contained in the compact set $\pi^{-1}(\{0\}\cup\bigcup_\nu\{y_\nu\})\cap A$, we would find a convergent subsequence $x_{\nu_\mu}\to x_0\in A\setminus\{0\}$. Then, necessarily, $p(x_0)=0$ which is a contradiction.
\end{proof}

Here is our generalization of Theorem \ref{Spodzieja}:

\begin{thm}\label{wzor}
Let $p\colon {\C}^m\times{\C}^n\ni (x,y)\mapsto y\in{\C}^n$ and let $A\subset {\C}^m\times{\C}^n$ be an irreducible $k$-dimensional locally analytic set. Assume that $p^{-1}(0)\cap A=\{0\}$. Then
$$
i(p^{-1}(0)\cdot A;0)=\tilde{m}_0(p|_A)\cdot \mathrm{deg}_0p(A).
$$
\end{thm}
\begin{proof}
By the assumption we may suppose that $A$ is an analytic subset of ${\C}^m\times D$ where $D\subset{\C}^n$ is a domain and the projection $p|_A$ is proper. Thus we have an irreducible $k$-dimensional analytic set $p(A)\subset D$ by the Remmert Proper Mapping Theorem. We may assume that the coordinates in ${\C}^n$ are chosen in such a way that the natural projection $\pi$ onto the first $k$ coordinates realizes the degree $\mathrm{deg}_0 p(A)$. Write $$T:=\{0\}^m\times\{0\}^k\times{\C}^{n-k}.$$



Let $\sigma\subset\mathrm{Reg} p(A)$ be the critical locus of the branched covering $p|_{A\setminus p^{-1}(\mathrm{Sng} A)}$ defined over the connected manifold $\mathrm{Reg} p(A)$; it has dimension strictly smaller than $k$. Actually, by The Andreotti-Stoll Theorem \cite{Loj} V 7.2, there is a nowheredense analytic subset $\Sigma\subset p(A)$ containing both $\mathrm{Sng} p(A)$ and $\sigma$. 

Since $0$ is the unique point in its fibre, we easily conclude that $\tilde{m}_0(p|_A)$ coincides with the covering number of $p$ over $\mathrm{Reg} p(A)$ (it suffices to use Lemma \ref{wlokno}).

Observe that we are dealing with a possibly improper isolated intersection $p^{-1}(0)\cap A=\{0\}$. In order to compute the intersection multiplicity according to \cite{ATW} we extend $p^{-1}(0)$ by $$N:= p^{-1}(0)\oplus T$$ We still have an isolated intersection $N\cap A=\{0\}$, since $p^{-1}(0)\cap A=\{0\}$. 

For the generic point $u\in{\C}^k$ near zero we have $\mathrm{deg}_0 A$ points in the fibre $\pi^{-1}(u)\cap p(A)$. Since $\pi(\mathrm{Sng} p(A))$ as well as $\pi(\sigma)$ are both nowheredense (since $\pi(\Sigma)$ is such), we may assume that $\pi^{-1}(u)\cap p(A)$ is contained in $\mathrm{Reg} p(A)\setminus\sigma$. Therefore, each of the points in the fibre considered splits up into $\tilde{m}_0(p|_A)$ points in the pre-image by $p|_A$. This shows that $m_0(\pi\circ p|_A)=\tilde{m}_0(p|_A)\cdot \mathrm{deg}_0 p(A)$. 

On the other hand, $\pi\circ p$ is just the natural projection along $N$. The intersection $N\cap A$ being proper we obtain $m_0(\pi\circ p|A)=i(N\cdot A;0)$. It remains to prove that $i(N\cdot A;0)=i(p^{-1}(0)\cdot A;0)$. By \cite{ATW} Theorem 4.4 this is true if only 
$$
C_0(A,p^{-1}(0))\cap N=p^{-1}(0).
$$
By definition we have $C_0(A,p^{-1}(0))=C_0(A-p^{-1}(0))$. But $p^{-1}(0)$ is linear, so $A-p^{-1}(0)=A+p^{-1}(0)=p(A)+p^{-1}(0)$ and the latter is just ${\C}^m\times p(A)$. Thus
$$
C_0(A,p^{-1}(0))={\C}^m\times C_0(p(A))
$$
and by the choice of $N$ the proof is accomplished.
\end{proof}
\begin{rem}
In the Theorem we do not need $A$ to be irreducible. As it follows from the proof, this assumption can be replaced by the assumption that $p(A)$ is irreducible.
\end{rem}
\begin{cor}
Theorem \ref{wzor} implies Theorem \ref{Spodzieja}.
\end{cor}
\begin{proof}
It is enough to apply Theorem \ref{wzor} to $A:=\Gamma_f$ and $p\colon D\times{\C}^n\to {\C}^n$, since we obviously have $f(D)=p(\Gamma_f)$, $\tilde{m}_0(p|_A)=\tilde{m}_0(f)$ and $i_0(f)=i(p^{-1}(0)\cdot \Gamma_f;0)$.
\end{proof}

In view of the pull-back problem, the following natural question arises. 
Suppose that $F\colon D\to U$ is a holomorphic branched covering between domains $U,V\subset{\C}^m$. We know that in such a case, the critical locus $\sigma\subset U$ coincides with the set of critical points $F(\{\mathrm{Jac} F=0\})$ and so by the Remmert Proper Mapping Theorem, is a hypersurface (\footnote{Note that since for the generic $x\in D$, $m_x(F)=1$, then $\mathrm{Jac} F\not\equiv 0$.}). Suppose that $V\subset D$ is an irreducible analytic set such that $V\cap \{\mathrm{Jac} F=0\}$ is nowheredense in $V$. Does it follow that $F^{-1}(\sigma)\cap V$ is nowheredense in $V$? 

We can ask the same question in a slightly more general setting. Namely, let the projection $\pi\colon A\to U$ be a branched covering over the domain $U\subset{\C}^k$ where $A$ is pure $k$-dimensional and let $\sigma\subset U$ be the critical locus. Put $\Sigma$ for the branching locus of $\pi$, i.e. $\mathrm{Sng} A$ together with those regular points at which $\pi$ is not surjective. If $V\subset A$ is an irreducible analytic set such that $V\cap \Sigma$ is nowheredense in $V$, does it follow that $V\cap \pi^{-1}(\sigma)$ is again nowheredense in $V$?

It is immediately clear that the answer to the second question is negative. Take for instance $A=\{x^2=y^2\}\cup \{y=1\}$ with $\pi(x,y)=x$ and $V=\{(0,1)\}$. Then over the unit disc $U$ we have $\sigma=\{0\}$, $\Sigma=\{(0,0)\}$ and $\pi^{-1}(\sigma)=\Sigma\cup V$.

At the same time this suggests a counter-example to the first question. Namely, let $F(x,y)=(x^2y,x+y)$. Since $F^{-1}(0)=\{(0,0)\}$, $F$ is proper in a neighboourhood of zero. We have $\mathrm{Jac} F(x,y)=(2x-y)x$, whence $\sigma=\{(0,y)\colon y\}\cup\{(4y^3,3y)\colon y\}$. Then take $V=\{(x,0)\colon x\}$. Clearly, $V\cap \{\mathrm{Jac} F=0\}=\{(0,0)\}$ but $V\subset F^{-1}(\sigma)$.

\section{Proof of Theorem \ref{pullback} via analytic intersection theory}

 We assume that $F\colon({\C}_z^m,0)\to({\C}_w^m,0)$ is a germ of a holomorphic finite mapping and $W\subset{\C}^m_w$ is an analytic germ at zero such that the pull-back $V:=F^{-1}(W)$ is a smooth germ (and hence irreducible). 

We make the following simplifications (one may find necessary details in e.g. \cite{Ch}):
\begin{enumerate}
\item There is a representant $F\colon D\to U$ defined on a domain $D\subset{\C}^m$, with $F^{-1}(0)=\{0\}$ and such that $F$ is a branched covering of the domain $U\subset{\C}^m$. Both $D$ and $U$ are arbitrarily small (our problem is local, shrinking the domains does not affect the multiplicity $\mu:=m_0(F)$ which is the sheet-number).

\item We may assume that $W$ is an analytic subset of $U$, whence $V$ is analytic in $D$. We posit that $V$ is smooth and irreducible.

\item Of course $F(V)=W$. The germ $V$ being irreducible, so is $W$ (it is even normal, see \cite{L}). Let $k=\dim V=\dim W$ (the dimension is pure). By a change of variables (and shrinking the domains), we may suppose that $V=({\C}^k\times\{0\}^{m-k})\cap D$. 

\item Moreover, we may assume that $U$ is of the form $U'\times U''\subset{\C}^k_{w'}\times{\C}^{m-k}_{w''}$ and the projection $\pi(w',w'')=w'$ realizes the local degree (Lelong number) $d:=\mathrm{deg}_0 W$. In particular, we must have $\pi^{-1}(0)\cap C_0(W)=\{0\}$.
\end{enumerate}

Observe one nice property:
\begin{lem}\label{np}
For any set $E$ such that $E\cap V=\{0\}$, one has $F(E)\cap W=\{0\}$.
\end{lem}
\begin{proof}Let $w\in F(E)\cap W$. Then there exists a point $e\in E$ such that $F(e)\in W$ and so $F^{-1}(F(e))\subset V$. In particular $e\in E\cap V=\{0\}$ and so $0=F(e)=w$.\end{proof}

\medskip
Our aim is to prove that $0\in\mathrm{Reg} W$ which means exactly that $d=1$. 

\medskip
Let $f:=F|_V$. It is a finite holomorphic mapping taking values in $W$. Now, observe that $W$ being irreducible, $\mathrm{Reg} W$ is connected, and so $f$ has a well-defined multiplicity $\lambda$ as a branched covering when restricted to $V':=V\setminus f^{-1}(\mathrm{Sng} W)$. Moreover, since $f^{-1}(0)=\{0\}$, by Lemma \ref{wlokno}, we conclude that the regular geometric multiplicity $\tilde{m}_0(f)$ coincides with $\lambda$. 

By Theorem \ref{wzor}, we have 
$$
d\cdot \lambda=i(\Gamma_f\cdot (D\times\{0\}^m);0)=i(\Gamma_f\cdot(V\times\{0\}^m);0)
$$
since $f$ can also be treated as a map ${\C}^k\supset V\to {\C}^m$.

Now, we will prove a counter-part of Lemma 4.2 from \cite{L} in two steps. Let us introduce the \textit{generic multiplicity of $F$ along} $V$ as
$$
m_V(F):=\min\{m_x(F)\mid x\in V\}.
$$
This makes sense for any holomorphic branched covering $F$ over a connected manifold and an irreducible analytic subset $V$ of the domain that can be a pure dimensional analytic set. 
First a simple general observation already made in \cite{L} Lemma 4.2.
\begin{lem}\label{set}
Let $\pi\colon A\to U$ be a holomorphic branched covering where $A$ is a pure $m$-dimensional analytic set and $U$ a domain in ${\C}^m$. Then for any analytic irreducible subset $V\subset A$, there exists an analytic nowheredense set $S\subsetneq V$ such that for all $x\in V\setminus S$, $m_x(\pi)=m_V(\pi)$.
\end{lem}
\begin{proof}
Since $A$ is pure dimensional, by 10.1 Lemma 1 in \cite{Ch} we know that the sets $A_s:=\{x\in A\mid m_x(\pi)\geq s\}$ are analytic. Then so are the sets $V\cap A_s$ and since $A=A_1\supset\ldots\supset A_\mu\supset A_{\mu+1}=\varnothing$ where $\mu$ is the sheet number, by the identity principle for analytic sets (see \cite{Ch}), we are done.
\end{proof}
\begin{prop}
If $F\colon D\to U$ is a holomorphic branched covering between two domains containing $0\in{\C}^m$, $F^{-1}(0)=\{0\}$ and $V\subset D$ is a analytic set irreducible at zero, then assuming that $F^{-1}(F(V))=V$, we have $$m_0(F)=\tilde{m}_0(F|_V)\cdot m_V(F).$$
\end{prop}
\begin{proof}
Once again using Lemma \ref{wlokno} we see that $\tilde{m}_0(F)$ is the covering number of $F|_{V\setminus F^{-1}(\mathrm{Sng} V)}$ over $\mathrm{Reg} V$. Of course, $m_0(F)$ is the covering number of $F$.

Let $S\subsetneq V$ be the set from the previous Lemma. Then there exists a point $x\in V$ such that $x\notin S$, $F(x)\in\mathrm{Reg}F(V)$ but it does not belong to the critical set of $F|_{V\setminus F^{-1}(\mathrm{Sng}F(V))}$. Since $F^{-1}(F(x))\subset V$ and since it consists of exactly $\tilde{m}_0(F)$ points $z$ such that for each of them, $m_{z}(F)=m_V(F)$, we conclude by the known Stoll formula $(St)$ mentioned earlier:
$$
m_0(F)=\sum_{z\in F^(F(x))}m_z(F)=\tilde{m}_0(F|_V)\cdot m_V(F)
$$
as required.
\end{proof}

Let us come back to our considerations. Write $\varkappa:=m_V(F)$. The formula obtained above yields
$$
\mu=\lambda\cdot \varkappa.
$$
Observe that $\mu=m_0(F)=i(\Gamma_F\cdot (D\times\{0\}^m);0)$ is the proper intersection index. Of course, $\Gamma_F\supset\Gamma_f$ and by \cite{ATW} Proposition 5.3 together with the definitions of \cite{ATW} Section 4, we obtain
$$
i(\Gamma_F\cdot (D\times\{0\}^m);0)\geq i(\Gamma_f\cdot (D\times\{0\}^m);0).
$$
Note that here the fact that $\Gamma_f$ is smooth is essential.

Now we need only the following lemma.
\begin{lem}
If $F\colon D\to U$ is a holomorphic branched covering between two domains containing $0\in{\C}^m$, $F^{-1}(0)=\{0\}$ and $V\subset D$ is an irreducible analytic set, then 
$$
m_V(F)=1\ \Leftrightarrow\ \mathrm{Jac} F|_V\not\equiv 0.
$$
\end{lem}
\begin{proof}
As earlier let $S\subsetneq V$ be the set from Lemma \ref{set}. Now, if $m_x(F)=1$ along $V\setminus S$, it means that $F$ is locally invertible at the points of $V\setminus S$ and so its Jacobian cannot vanish there. 

On the other hand, if the set $Z:=\{x\in V\mid \mathrm{Jac} F(x)= 0\}$ does not coincide with $V$, it is nowheredense in $V$ (since $V$ is irreducible). But then for any $x\in V\setminus Z$, $F$ is invertible at $x$, and so there must be $m_x(F)=1$. Whence $m_V(F)=1$.
\end{proof}

Eventually, from all the preceding discussion we obtain
\begin{align*}
\tilde{m}_0(f)\cdot m_V(F)&=\lambda\cdot\varkappa=\\
&=\mu=m_0(F)=i(\Gamma_F\cdot (D\times\{0\}^m);0)\geq\\
&\geq i(\Gamma_f\cdot (D\times\{0\}^m);0)=\lambda\cdot d=\\
&=\tilde{m}_0(f)\cdot\mathrm{deg}_0 W.
\end{align*}
In Greek lettering, 
$$
\lambda\cdot\varkappa\geq \lambda\cdot d,
$$
whence $\varkappa\geq d$. But under the assumptions of Theorem \ref{pullback} we have, by the last Lemma, $\varkappa=1$ which ends the proof of this Theorem.

\section{Acknowledgements}

There is a very long list of people whom the pull-back problem was discussed with. First we should thank the late Marco Brunella 
who showed great interest in the matter. Then in chronological order we acknowledge our indebtedness to Lucy Moser-Jauslin, Carlo Perrone and Ewa Cygan. 

The author is grateful to Stanis\l aw Spodzieja for his remarks on formula $(S)$.

During the preparation of this note, the author was partially supported by Polish Ministry of Science and Higher Education grant  IP2011 009571.

\end{document}